\documentclass{amsart}

\usepackage[wheretonumber = subsection]{intropapers}
\usepackage[margin=1in]{geometry}
\usepackage{cite}
\usepackage{hyperref}
\usepackage{url}

\title{Canonical $\beta$-extensions}
\author{Andrea Dotto}
\date{}

\begin{document}

\maketitle

\begin{abstract}
We compare the level zero part of the type of a representation of~$\GL(n)$ over a local non-archimedean field with the tame part of its Langlands parameter restricted to inertia. By normalizing this comparison, we construct canonical $\beta$-extensions of maximal simple characters.
\end{abstract}
\setcounter{tocdepth}{1}
\tableofcontents
\section{Introduction.}
Let~$F$ be a non-archimedean local field with residue field~$\mbf$ of characteristic~$p$. 
The category of smooth representations of~$\GL_n(F)$ with complex coefficients is a central object of interest in the classical local Langlands correspondence, and has been thoroughly studied by many authors.
As a result, after the fundamental work of Bushnell--Kutzko~\cite{BKbook} and Bernstein--Zelevinsky~\cite{BZsurvey, BZinduced, Zelevinskyinduced} there is a complete classification of its irreducible objects.
This classification is very explicit, but some parts of it depend on choices which can prove ambiguous in applications and do not always have a clear interpretation in Galois-theoretic terms across the local Langlands correspondence.
The aim of this short note is to resolve one of these ambiguities by specifying a choice of so-called \emph{$\beta$-extensions} which has good properties with respect to functorial procedures such as parabolic induction, the local Langlands correspondence and the Jacquet--Langlands correspondence.

To put our results in context we begin by recalling a consequence of a theorem of Bushnell and Henniart, namely~\cite[Types Theorem]{BHeffective}, which gives a description of the maximal simple types contained in an irreducible supercuspidal~$\GL_n(F)$ representation~$\pi$ in terms of its Langlands parameter~$\rec(\pi)$.
More precisely, we will work with the ``level zero part" of~$\rec(\pi)$, which is a Galois conjugacy class of characters~$[\chi]$ of the multiplicative group of a finite extension~$\be_{n/\delta(\Theta_F)}$ of the residue field of~$F$, defined in Section~\ref{Langlandsparameters}.
This construction is elementary in nature but quite involved, and amounts to an instance of Clifford theory for the Weil group~$W_F$.
On the other hand, let $\theta: H^1_\theta \to \bC^\times$ be a maximal simple character contained in~$\pi$, and write~$\Theta_F$ for its endo-class.
Then~\cite[Types Theorem]{BHeffective} implies that there exists a unique $\beta$-extension~$\kappa_{n, F}$ of~$\theta$ such that the representation
\[
\bK_{\kappa_{n, F}}(\pi) = \Hom_{J^1_\theta}(\kappa_{n, F}, \pi)
\]
corresponds to~$[\chi]$ under the Green parametrization (or equivalently, Deligne--Lusztig induction).
Varying~$n$ amongst multiples of the degree of~$\Theta_F$, we get a conjugacy class~$\kappa_{m, F}$ of $\beta$-extensions of maximal simple characters in~$\GL_m(F)$ of endo-class~$\Theta_F$ whenever the degree of~$\Theta_F$ divides~$m$. 

In order to make these statements precise one should provide definitions of all the objects involved and normalize various choices: we will do this to the extent needed in the main body of the article, referring to~\cite{InertialJL} for a full exposition.
Amongst other things, one should recall that~$\Hom_{J^1_\theta}(\kappa_{n, F}, \pi)$ is a representation of a group~$J_\theta/J^1_\theta$ which is isomorphic to a finite general linear group~$\GL_{n/\delta(\Theta_F)}(\be)$, allowing one to define the Green parametrization of its irreducible representations.
(Actually, we will not work with~$\bK_{\kappa_{n, F}}(\pi)$ but directly with its Green parameter: see Section~\ref{invariantsG}, Remark~\ref{KLambda}).
Our main result can be understood as a generalization of this property of~$\kappa_{n, F}$ to essentially square-integrable representations, or more generally to representations which are~\emph{simple}, in the sense that their supercuspidal support is inertially equivalent to one of the form~$[\GL_{n/r}(F)^{\times r}, \pi_0^{\otimes r}]$.

\begin{thm}\label{BHgeneralize}
(Theorem~\ref{compatiblecanonical}.)
Let~$\pi$ be an irreducible smooth representation of~$\GL_n(F)$ whose supercuspidal support is inertially equivalent to~$[\GL_{n/r}(F)^{\times r}, \pi_0^{\otimes r}]$.
Then the supercuspidal support of every Jordan--H\"older factor of~$\bK_{\kappa_{n, F}}(\pi)$ is a multiple of~$\bK_{\kappa_{n/r, F}}(\pi_0)$.
\end{thm}

Notice that in this case we have $\rec(\pi)|_{I_F} \cong \rec(\pi_0)^{\oplus r}|_{I_F}$, and~$\rec(\pi_0)$ is an irreducible representation with a level zero part~$[\chi_0]$. 
Then Theorem~\ref{BHgeneralize} relates~$[\chi_0]$ to a type-theoretic invariant of~$\pi$, namely the supercuspidal support of~$\bK_{\kappa_{n, F}}(\pi)$, and in this sense it generalizes~\cite[Types Theorem]{BHeffective} to non-cuspidal simple representations.

An equivalent way of stating Theorem~\ref{BHgeneralize} is that~$\kappa_{n, F}$ and~$\kappa_{n/r, F}$ are \emph{compatible} in the sense of~\cite[Section~3.3]{InertialJL}, and this is the statement that we actually prove in Theorem~\ref{compatiblecanonical}.
As~$n$ varies, the representations~$\kappa_{n, F}$ can therefore be understood as forming a family of $\beta$-extensions canonically attached to~$\Theta_F$, by virtue of their compatibility with parabolic induction and the local Langlands correspondence.
Furthermore, in Section~\ref{innerforms} we apply the main results of~\cite{InertialJL} to extend this family to all inner forms of~$\GL_n(F)$ and prove compatibility with the Jacquet--Langlands correspondence. 
See Theorem~\ref{betainnerforms}.

We end this introduction by discussing the relationship of this note with other works in the literature.
The lack of a canonical choice of~$\beta$-extensions has been noticed in several contexts, for instance~\cite{Blondelbeta} and~\cite{BHSsymplectic}, and our method provides a complete solution to this problem for inner forms of~$\GL_n$.
We emphasize that the compatibility of two given $\beta$-extensions in two different general linear groups (such as~$\GL_n(F)$ and~$\GL_{n/r}(F)$) could presumably be addressed by explicit calculations using only the constructions of~\cite{BKbook}: the main contribution of this note is another method to address this compatibility, which works for all inner forms of~$\GL_n(F)$, uniformly in~$n$ and without needing any calculation. 
It consist in adapting a technique from~\cite{SecherreStevensJL, InertialJL} involving congruences modulo a prime~$\ell \ne p$.
The main input is the verification of certain formal and numerical properties of the local Langlands correspondence with respect to reduction mod~$\ell$, which we do by applying work of Vign\'eras \cite{Vignerasl2, Vignerasl1}. 
As a result, it is quite likely that the method would apply to other groups if similar properties of the local Langlands correspondence were established. 

\subsection{Acknowledgments.} This paper grew from a comment of Vincent S\'echerre and Shaun Stevens regarding~\cite{InertialJL}, namely that $p$-primary $\beta$-extensions in different groups might not be compatible. I am grateful to them for this remark. The idea that the local Langlands correspondence might be used to normalize $\beta$-extensions was suggested by Colin Bushnell. This work was supported by the Engineering and Physical Sciences Research Council [EP/L015234/1], The EPSRC Centre for Doctoral Training in Geometry and Number Theory (The London School of Geometry and Number Theory), University College London, and Imperial College London.

\subsection{Notation and conventions.} 
To keep this note concise we will not dwell for long on type-theoretic background but refer the reader to standard references such as~\cite{BKbook, BHeffective} for a detailed treatment of the theory.
Our notation in this paper will mostly follow~\cite{InertialJL}.
We let~$F$ denote a nonarchimedean local field, $\mbf$ the residue field of~$F$, $F_n$ the unramified extension of~$F$ of degree~$n$ in some fixed algebraic closure~$\overline{F}$ of~$F$, and~$\mbf_n$ the residue field of~$F_n$. 
The group of Teichm\"uller roots of unity in~$F$ is denoted~$\mu_F$. 
We write~$W_F$ for the Weil group of~$\lbar{F}/F$, $I_F$ for the inertia group and $P_F$ for the wild inertia group. 
We normalize the Artin map $\Art_F: F^\times \to W_F^{\mathrm{ab}}$ so that uniformizers correspond to geometric Frobenius elements. 
If~$\sigma$ is a representation of~$W_F$, its twist by the unramified character of~$W_F$ sending a geometric Frobenius element to~$q^{-n}$ for~$n \in \bZ$ is denoted~$\sigma(n)$. 
This character for~$n = 1$ corresponds to the normalized absolute value of~$F$ under~$\Art_F$, hence we denote it by $w \mapsto |w|$. 

For a prime number~$\ell$, we say that an element~$g$ of a finite group is $\ell$-primary if it has order a power of~$\ell$ and $\ell$-regular if it has order coprime to~$\ell$. 
We write $g^{(\ell)}$ for the $\ell$-regular part of~$g$ and~$g_{(\ell)}$ for the $\ell$-primary part of~$g$, so that~$g = g^\pexp{\ell}g_\pexp{\ell}$.

Representations of a locally profinite group like $\GL_n(F)$ or~$W_F$ are assumed to be smooth (and finite-dimensional for~$W_F$), with coefficients over an algebraically closed field~$R$ of characteristic different from~$p$, which will be specialized to~$\bC$, $\cbQ_\ell$ and~$\cbF_\ell$ in the course of the paper. 
When $R = \cbQ_\ell$ and~$V$ is a $\cbQ_\ell$-integral representation of finite length we write~$\br_\ell(V)$ for the semisimplified mod~$\ell$ reduction in the sense of~\cite[Section~1.4]{Vignerasl2}.
Parabolic induction from a standard Levi subgroup is always normalized and taken along the upper-triangular parabolic, and we write $\pi_1 \times \cdots \times \pi_n$ for the parabolic induction of $\pi_1 \otimes \cdots \otimes \pi_n$. 
Working with normalized induction requires us to fix a square root of~$q$ in~$R^\times$, but changing it does not modify the inertial class of the supercuspidal support of any given irreducible representation, hence the choice will not affect any of our results about inertial classes.

\section{Invariants of representations of $\GL_n(F)$.}\label{invariantsG}
Let~$R$ be an algebraically closed field of characteristic different from~$p$.
The irreducible smooth representations of~$R[\GL_n(F)]$ are partitioned according to the blocks of the category of smooth representations, which are parametrized by inertial classes of supercuspidal supports, following work of Bernstein~\cite{BernsteinDeligne} (when~$R$ has characteristic zero) and S\'echerre--Stevens~\cite{SecherreStevensblocks} (when~$R$ has positive characteristic).
By definition, an inertial class is \emph{simple} if it contains a representative of the form
\[
\left [ \GL_{n/r}(F)^{\times r}, \pi_0^{\otimes r} \right ]
\]
for an irreducible supercuspidal $R[\GL_{n/r}(F)]$-representation~$\pi_0$.

In~\cite{InertialJL} we have written down a parametrization of simple inertial classes~$\fs$ in terms of two type-theoretic invariants, which we recall briefly.
The first one is denoted~$\cl(\fs)$ and consists of an endo-class of simple characters defined over~$F$.
To any endo-class~$\Theta_F$ there are associated the following objects: 
\begin{itemize}
\item a finite unramified extension~$E/F$ contained in our fixed algebraic closure~$\lbar F / F$, called the unramified parameter field of~$\Theta_F$. The degree~$[E:F]$ is denoted~$f(\Theta_F)$.
\item a number~$\delta(\Theta_F)$, called the degree of~$\Theta_F$, which is divisible by~$[E:F]$. We also introduce~$e(\Theta_F) = \delta(\Theta_F)/f(\Theta_F)$.
\item when~$\Theta_F$ is the endo-class of a maximal simple character in~$\GL_n(F)$, the degree~$\delta(\Theta_F)$ divides~$n$ and we have the finite cyclic Galois group~$\Gal(\be_{n/\delta(\Theta_F)}/\be)$, which we denote~$\Gamma(\Theta_F)$. It acts on the set of $R$-valued characters of~$\be_{n/\delta(\Theta_F)}^\times$, which we denote~$X_R(\Theta_F)$.
\end{itemize}
The second invariant of simple inertial classes, called the level zero part, depends on the choice of a lift $\Theta_E \to \Theta_F$ of~$\Theta_F = \cl(\fs)$ to an endo-class defined over~$E$ and on the choice of a conjugacy class~$\kappa$ of $\beta$-extensions of maximal simple characters in~$\GL_n(F)$ of endo-class~$\Theta_F$.
(We will also refer to such objects as \emph{maximal $\beta$-extensions} of endo-class~$\Theta_F$.)
The invariant associates to a simple inertial class of $R[\GL_n(F)]$-representations of endo-class~$\cl(\fs) = \Theta_F$ an orbit of~$\Gamma(\Theta_F)$ on $X_R(\Theta_F)$, denoted
\[
\Lambda(\fs, \Theta_E, \kappa) \in \Gamma(\Theta_F) \backslash X_R(\Theta_F).
\]

By~\cite[Formulas~(3.5) and~(3.6)]{InertialJL}, the map~$\Lambda(-, \Theta_E, \kappa) $ depends in a simple way on the choice of~$\Theta_E$ and~$\kappa$.
By~\cite[Theorem~3.21]{InertialJL}, the pair $(\cl(\fs), \Lambda(\fs, \Theta_E, \kappa))$ determines the simple inertial class~$\fs$ uniquely.

\begin{rk}\label{KLambda}
Let~$\pi$ be an irreducible representation in the inertial class~$\fs$. 
We make explicit the connection of~$\Lambda(\fs, \Theta_E, \kappa)$ to the representation~$\bK_\kappa(\pi)$ that occurs in the introduction.
Assume first that~$\pi$ is supercuspidal.
The group~$J_\theta/J^1_\theta$ is noncanonically isomorphic to a finite general linear group, and the role of the lift~$\Theta_E \to \Theta_F$ is to select an inner conjugacy class of isomorphisms
\[
\Psi(\Theta_E): J_\theta/J^1_\theta \isom \GL_{n/\delta(\Theta_F)}(\be).
\]
Then~$\bK_\kappa(\pi)$ can be unambiguously identified with a representation of~$\GL_{n/\delta(\Theta_F)}(\be)$, which turns out to be irreducible and supercuspidal.
By the Green parametrization, it corresponds to an element of~$\Gamma(\Theta_F)\backslash X_R(\Theta_F)$, which is~$\Lambda(\fs, \Theta_E, \kappa)$ by definition.

Assume now that the supercuspidal support of~$\pi$ is inertially equivalent to~$[\GL_{n/r}(F)^{\times r}, \pi_0^{\otimes r}]$.
Then~$\cl(\pi) = \cl(\pi_0)$, and by~\cite[Proposition~3.13, Definition~3.14]{InertialJL} there exists a unique conjugacy class~$\kappa_0$ of maximal $\beta$-extensions in~$\GL_{n/r}(F)$ that is compatible with~$\kappa$, in the sense that for all irreducible representations~$\pi_0$ of~$\GL_{n/r}(F)$ there is an isomorphism
\[
\bK_{\kappa}(\pi_0^{\times r}) \isom \bK_{\kappa_0}(\pi_0)^{\times r}.
\]
(See~\cite[Definition~3.11]{InertialJL} for more details on the definition of compatibility.)
By definition, $\Lambda(\pi, \Theta_E, \kappa)$ is the inflation of $\Lambda(\pi_0, \Theta_E, \kappa_0)$ to~$\be_{n/\delta(\Theta_F)}^\times$ via the norm $\be_{n/\delta(\Theta_F)}^\times \to \be_{n/r\delta(\Theta_F)}^\times$.
\end{rk}

\section{Invariants of Weil--Deligne representations.}\label{Langlandsparameters}
In this section we are concerned with establishing analogues for Weil--Deligne representations of the invariants of Section~\ref{invariantsG}.
We do so in order to give a precise definition of~$\kappa_{n, F}$ and to state its properties within the framework of \cite{SecherreStevensJL, InertialJL}.

\subsection{Local Langlands correspondence.}
To start with, we briefly review the local Langlands correspondence for~$\GL_n(F)$.
The Langlands parameters for~$\GL_n(F)$ can be identified with Frobenius-semisimple Weil--Deligne representations, which are pairs $(V, N)$ consisting of a semisimple smooth representation of~$W_F$ and a nilpotent monodromy operator $N: V(1) \to V$. They can be written uniquely as direct sums
\begin{displaymath}
V = \bigoplus_i \sigma_i \otimes \Sp(n_i)
\end{displaymath} 
for irreducible smooth representations~$\sigma_i$ of~$W_F$. The special representation~$\Sp(n)$ has a basis $\{ e_1, \cdots, e_n\}$ such that $we_i = |w|^{i-1}e_i$ for~$w \in W_F$, and the monodromy acts as $Ne_i = e_{i+1}$ for $i \in \{1, \ldots, n-1 \}$.

The local Langlands correspondence is a bijection, denoted~$\rec$, of the isomorphism classes of irreducible smooth representations of~$\bC[\GL_n(F)]$ onto the complex Frobenius-semisimple Weil--Deligne representations of dimension~$n$. 
It restricts to a bijection~$\rec^0$ from supercuspidal irreducible representations to irreducible smooth $\bC[W_F]$-representations (notice that since the kernel of~$N$ is stable under~$W_F$, these have trivial monodromy). 
Amongst the many properties of~$\rec$, we will need its compatibility with the Bernstein--Zelevinsky classification.

To state this compatibility, recall that a \emph{segment} of supercuspidal representations of~$\bC[\GL_{n}(F)]$, of length~$m$, consists of a sequence 
\begin{displaymath}
(\rho, \rho(1), \ldots, \rho(m-1))
\end{displaymath} 
of twists of a supercuspidal~$\rho$ by powers of the unramified character $g \mapsto |\det(g)|$. 
The Bernstein--Zelevinsky classification is a bijection of $\coprod_{m \geq 0} \Irr \bC[\GL_m(F)]$ with the set of multisets of segments of supercuspidal representations.
Assume that~$\pi \in \Irr \bC[\GL_n(F)]$ corresponds to the multiset~$\{ \Delta_1, \ldots, \Delta_r\}$, where $\Delta_i = (\rho_i, \ldots, \rho_i(n_i - 1))$. 
Then the construction of~$\rec$ from~$\rec^0$ implies that $\rec(\pi) = \bigoplus_i \rec(\rho_i) \otimes \Sp(n_i)$. 
Hence, if~$\pi$ has supercuspidal support
\begin{equation}\label{Langlandssupport}
[\GL_{n_1}(F) \times \cdots \times \GL_{n_r}(F), \pi_1 \otimes \cdots \otimes \pi_r]
\end{equation}
then the $W_F$-representation underlying~$\rec(\pi)$ is $\rec(\pi_1) \oplus \cdots \oplus \rec(\pi_r)$.
This property of~$\rec$ implies that the inertial class of an irreducible representation of~$\GL_n(F)$ is described by the restriction to inertia of its Langlands parameter, in the following sense. 
For a Weil--Deligne representation~$\tau$, write~$\tau|I_F$ to denote the restriction to~$I_F$ of the underlying $W_F$-representation.
Let~$\pi_1, \pi_2$ be two irreducible smooth representations of~$\bC[\GL_n(F)]$.  
Then $\rec(\pi_1)|_{I_F} \cong \rec(\pi_2)|_{I_F}$ if and only if~$\pi_1$ and~$\pi_2$ are inertially equivalent.

In the course of this paper will also need to work over~$\cbQ_\ell$ for primes~$\ell \not = p$. 
To do so, we can fix a ring isomorphism~$\iota_\ell: \bC \to \cbQ_\ell$ and then transfer~$\rec$ to a bijection~$\rec_\ell$ from irreducible smooth $\cbQ_\ell[\GL_n(F)]$-representations to Frobenius-semisimple Weil--Deligne representations of dimension~$n$ over~$\cbQ_\ell$. 
Some care has to be taken here since the Langlands correspondence does not commute with all automorphisms of~$\bC$, and one way of getting around this is to fix a square root of~$q$ in~$\bC$ and~$\cbQ_\ell$ and to work with isomorphisms~$\iota_\ell$ that preserve it. 
However, any two choices of~$\iota_\ell$ define bijections~$\rec_\ell$ which differ at most by a quadratic unramified twist at any given $\cbQ_\ell$-representation of~$\GL_n(F)$. 
Since we will mostly be concerned with the restriction to inertia of Weil--Deligne representations, our results will be independent of the choice of~$\iota_\ell$.

\subsection{Endo-classes.}
Bushnell and Henniart have shown how to attach an endo-class to an irreducible Weil group representation, making use of the following result. 
Write~$P_F^\vee$ for the set of complex irreducible smooth representations of the wild inertia group~$P_F \subset G_F$, and~$\mathcal{E}(F)$ for the set of endo-classes of simple characters over~$F$. 
There is a left action of~$W_F$ on~$P_F^\vee$ by conjugation. 
If~$\sigma$ is an irreducible representation of~$W_F$, then let~$r_F^1(\sigma) \in W_F \backslash P_F^\vee$ be the orbit contained in the restriction $\sigma |_{P_F}$ (which need not be multiplicity-free).

\begin{thm}[See~\cite{BHeffective}, Ramification Theorem]\label{ramification} 
The Langlands correspondence induces a bijection 
\begin{displaymath}
\Phi_F : W_F \backslash P_F^\vee \to \mathcal{E}(F)
\end{displaymath}
such that $\Phi_F(r^1_F(\sigma))$ is the endo-class of~$\rec^{-1}(\sigma)$ for any irreducible~$\sigma$. If $\gamma : F \to F$ is a topological automorphism, extended in some way to an automorphism of~$W_F$, then $\Phi_F(\gamma^*[\alpha]) = \gamma^*\Phi_F[\alpha]$ for all $[\alpha] \in W_F \backslash P_F^\vee$.
\end{thm}

The Ramification Theorem holds with coefficients in $\cbQ_\ell$, because any isomorphism $\iota_\ell: \bC \isom \cbQ_\ell$ induces via~$\rec_\ell$ a bijection between endo-classes for~$F$ over~$\cbQ_\ell$ and orbits of~$W_F$ on irreducible smooth $\cbQ_\ell$-representations of~$P_F$. This bijection is independent of the choice of~$\iota_\ell$.

Since~$P_F$ is a pro-$p$ group, the orbits of its irreducible smooth $\cbF_\ell$-representations under~$W_F$ are identified with those over~$\cbQ_\ell$ by choosing a lattice (which will be unique up to homothety) and reducing mod~$\ell$ (the reduction will be irreducible). 
Similarly, the endo-classes over~$\cbQ_\ell$ are identified with those over~$\cbF_\ell$, and the Ramification Theorem also holds over~$\cbF_\ell$.

\subsection{Level zero maps.}\label{levelzeroparameters}
By definition, a \emph{supercuspidal inertial type} for~$W_F$ is the restriction to inertia of an irreducible representation~$\sigma$ of~$W_F$. 
In this section, we use Clifford theory for the group~$W_F$ over the algebraically closed field~$R$ (of characteristic different from~$p$), as in~\cite{Vignerasl1} and Section~1 of~\cite{BHeffective}, to define the level zero part of a supercuspidal inertial type. 

Let~$\sigma$ be an irreducible smooth $R[W_F]$-representation of dimension~$n$. 
Since~$P_F$ is a normal subgroup of~$W_F$, the restriction $\sigma|_{P_F}$ is semisimple and consists of a single~$W_F$-orbit of irreducible representations (possibly with multiplicity). 
Let~$\alpha$ be a representative of this $W_F$-orbit, which we will denote~$[\alpha]_F$. 
Let~$T = T_\alpha$ be the tamely ramified extension of~$F$ corresponding to the stabilizer in~$W_F$ of the isomorphism class of~$\alpha$. 
It is a subfield of~$\lbar F$.

By~\cite[1.3]{BHeffective}, there exists a unique extension~$\rho_\alpha$ of~$\alpha$ to~$I_T$ with $p$-primary determinant, and~$\rho_\alpha$ extends to~$W_T$. 
We denote by~$\rho(\alpha)$ an arbitrary choice of extension of~$\rho_\alpha$ to~$W_T$. 
As in~\cite[Section~2.6]{Vignerasl1}, there exists a unique tamely ramified representation $\sigmatr(\alpha)$ of~$W_T$, denoted~$\tau$ in~\cite{BHeffective}, such that $\sigma \cong \Ind_T^F(\rho(\alpha) \otimes \sigmatr(\alpha))$. 

\begin{lemma}\label{alphaisotypic}
The $\alpha$-isotypic component of~$\sigma$, denoted~$\sigma_\alpha$, is isomorphic to $\rho(\alpha) \otimes \sigma^\tr(\alpha)$ as a representation of~$W_T$.
\end{lemma}
\begin{proof}
Notice that $\rho(\alpha) \otimes \sigmatr(\alpha)$ is an irreducible $W_T$-subspace of~$\sigma_\alpha$.
Let~$\{g_i\}$ be a set of representatives of~$W_F/W_T$ in~$W_F$.
Then $g_i(\rho(\alpha) \otimes \sigmatr(\alpha)) \subset \sigma_{g_i\alpha}$, hence
\[
R[W_F](\rho(\alpha) \otimes \sigmatr(\alpha)) = \bigoplus_i g_i (\rho(\alpha) \otimes \sigmatr(\alpha))
\]
and $R[W_F](\rho(\alpha) \otimes \sigmatr(\alpha))$ would be a proper $W_F$-subspace of~$\sigma$ if $\rho(\alpha) \otimes \sigmatr(\alpha)$ were properly contained in~$\sigma_\alpha$. 
\end{proof}

The representation~$\sigmatr(\alpha)$ can be written uniquely as an induced representation $\Ind_{T_d}^T(\chi_1(\alpha))$ for some unramified extension $T_d/T$ of degree~$d>0$ and some $\Gal(T_d/T)$-orbit of $T$-regular characters~$[\chi_1(\alpha)]$ of~$T_d^\times$ such that $\chi_1(\alpha)$ is trivial on the $1$-units $U^1(T_d)$. 
(We regard~$\chi_1(\alpha)$ as a character of~$W_{T_d}$ via the Artin reciprocity map $\Art_{T_d}^{-1}: W_{T_d} \to T_d^\times.$)
We find that $\sigma \cong\Ind_{T_d}^F(\rho_d(\alpha) \otimes \chi_1(\alpha))$ for the restriction~$\rho_d(\alpha)$ of~$\rho(\alpha)$ to~$W_{T_d}$. 

\begin{rk}\label{recoverisotypic}
Write~$\chi(\alpha) = \chi_1(\alpha) |\mu_{T_d}$. Let~$[\chi(\alpha)]$ be its orbit under~$\Gal(T_d/T)$. The restriction of~$\sigma_\alpha$ to~$I_{T_d} = I_T$ is a direct sum of the twists $\rho_\alpha \otimes \chi$ for~$\chi \in [\chi(\alpha)]$, hence we can recover~$[\chi(\alpha)]$ from~$\sigma$ as follows.
Take the $\alpha$-isotypic component~$\sigma_\alpha$ and restrict it to~$I_{T_\alpha}$. 
The restriction will decompose as a direct sum of twists of~$\rho_\alpha$ (which is the only irreducible extension of~$\alpha$ to~$I_{T_\alpha}$ with $p$-primary determinant character) by characters of~$\mu_{T_d}$.
Since $\mu_{T_d}$ has order coprime to~$p$ and $\dim \rho_\alpha$ is a power of~$p$, the map $\chi \mapsto \rho_\alpha \otimes \chi$ is injective, and this determines~$[\chi(\alpha)]$ as the set of characters such that~$\rho_\alpha \otimes \chi$ is a constituent of~$\sigma_\alpha | I_{T_\alpha}$.
\end{rk}

By Theorem~\ref{ramification}, the $W_F$-orbit $[\alpha]_F$ defines an endo-class~$\Theta_F = \Phi_F[\alpha]_F$. 

\begin{lemma}\label{dimensioncomputation}
We have the equality~$d = n/\delta(\Theta_F)$.
\end{lemma}
\begin{proof}
By \cite[Tame Parameter Theorem]{BHeffective}, the field~$T$ is isomorphic over~$F$ to a tame parameter field for~$\Theta_F$, and the degree~$\delta(\Theta_F)$ equals~$[T:F]\dim \alpha$. 
One the other hand, $\sigma$ decomposes as the direct sum of its $\alpha$-isotypic components for $\alpha \in [\alpha]_F$. 
Since the orbit $[\alpha]_F$ has~$[T:F]$ elements, Lemma~\ref{alphaisotypic} implies that we have the equality 
\[
n = [T:F](\dim \alpha)(\dim \sigmatr(\alpha)). 
\]
Hence $d = \dim \sigmatr(\alpha) = n/\delta(\Theta_F)$.
\end{proof}

Let us introduce the maximal unramified extension~$E = \Tur_\alpha$ of~$F$ in~$T_\alpha$. 
This is independent of the choice of~$\alpha$, and by \cite[Tame Parameter Theorem]{BHeffective} it is the unramified parameter field of~$\Theta_F$ in~$\overline{F}$. 
At this stage, we have attached to~$\sigma$ an endo-class~$\Theta_F$ of degree dividing~$n = \dim(\sigma)$, and whenever we choose a representative~$\alpha$ of the orbit~$[\alpha]_F$ attached to~$\Theta_F$, we obtain a $\Gal(\be_{n/\delta(\Theta_F)}/\be)$-orbit~$[\chi(\alpha)]$ of~$\be$-regular characters of~$\be_{n/\delta(\Theta_F)}^\times$, since $\mu_T = \mu_E$, $\mu_{T_d} = \mu_{E_d}$ and~$d = n/\delta(\Theta_F)$.
This is the same set of data occurring in Section~\ref{invariantsG}.
To complete the connection, we consider how~$[\chi(\alpha)]$ changes when we change representative $\alpha \in [\alpha]_F$. 

\begin{lemma}\label{changerepresentative}
Let~$g \in W_F$. Then~$[\chi(\ad(g)^*\alpha)]$ only depends on the image of~$g$ in~$W_F/W_E \cong \Gal(E/F)$.
\end{lemma}
\begin{proof}
By our explicit description of~$[\chi(\alpha)]$ in terms of the $\alpha$-isotypic component of~$\sigma$ (see Remark~\ref{recoverisotypic}) it follows that  $g^*[\chi(\ad(g)^*\alpha)] = [\chi(\alpha)]$. However, by definition, the group~$W_E$ fixes the $\Gal(E_d/E)$-conjugacy classes of characters of~$\mu_{E_d}$, for every~$d$.
\end{proof}

By Theorem~\ref{ramification}, a lift~$\Theta_E$ of~$\Theta_F$ to~$E$ defines an orbit~$[\alpha]_E$ of~$W_E$ on~$[\alpha]_F$. 
Hence we can define a set of characters
\[
\Lambda^+(\sigma, \Theta_E) \in \Gamma(\Theta_F) \backslash X_R(\Theta_F)
\]
by setting~$\Lambda^+(\sigma, \Theta_E) = [\chi(\alpha)]$ for any~$\alpha$ such that~$\Theta_E = \Phi_E[\alpha]_E$. 
By Lemma~\ref{changerepresentative} this is well-defined.
The behaviour of this level zero map~$\Lambda^+(-, \Theta_E)$ under change of lifts is the same as for~$\GL_n(F)$.

\begin{lemma}\label{changelift}
Let~$\gamma \in \Gal(E/F)$. Then $\gamma^* \Lambda^+(-, \gamma^*\Theta_E) = \Lambda^+(-, \Theta_E)$.
\end{lemma}
\begin{proof}
By Theorem~\ref{ramification}, if~$\Theta_E = \Phi_E[\alpha]$ then $\gamma^*\Theta_E = \Phi_E(\ad(g)^*[\alpha])$ for any lift~$g \in W_F$ of~$\gamma$. We have seen in the proof of Lemma~\ref{changerepresentative} that $g^*[\chi(\ad(g)^*\alpha)] = [\chi(\alpha)]$, which implies the lemma.
\end{proof}

Now we prove that the endo-class and the level zero part of a supercuspidal inertial type determine it uniquely, in analogy with the case of~$\GL_n(F)$.

\begin{pp}\label{sameinertialparameter}
Let~$\sigma_1, \sigma_2$ be irreducible $R[W_F]$-representations. 
Assume that~$\sigma_1, \sigma_2$ have the same endo-class~$\Theta_F$ under Theorem~\ref{ramification}.
Fix a lift~$\Theta_E \to \Theta_F$. 
Then $\Lambda^+(\sigma_1, \Theta_E) = \Lambda^+(\sigma_2, \Theta_E)$ if and only if $\sigma_1 |_{I_F} \cong \sigma_2 |_{I_F}$.
\end{pp}
\begin{proof}
Choose a representative~$\alpha$ of the $W_{E}$-orbit of representations of~$P_F$ attached to~$\Theta_E$.
By Remark~\ref{recoverisotypic}, the restriction to~$I_{T_\alpha}$ of the isotypic component~$\sigma_{i, \alpha}$ is 
\begin{equation}\label{restrictiontoinertia}
\sigma_{i, \alpha} | I_{T_\alpha} \cong \rho_\alpha \otimes \bigoplus_{\xi \in \Lambda^+(\sigma_i, \Theta_E)}\xi.
\end{equation}
Hence $\sigma_i|_{I_F}$ determines~$\Lambda^+(\sigma_i, \Theta_E)$, since it determines the isomorphism class of $\sigma_{i, \alpha}|_{I_{T_\alpha}}$.
For the converse, the Mackey formula
\begin{equation}\label{Mackeyformula}
\Res^{W_F}_{I_F}\sigma = \Res^{W_F}_{I_F}\Ind_{W_{T_\alpha}}^{W_F}\sigma_\alpha = \bigoplus_{\gamma  \in W_{T_\alpha} \backslash W_F / I_F} \Ind_{\gamma^{-1}I_{T_{\alpha}}\gamma}^{I_F}\Res^{\gamma^{-1}W_{T_{\alpha}}\gamma}_{\gamma^{-1}I_{T_{\alpha}}\gamma}\ad(\gamma)^*\sigma_{\alpha}.
\end{equation}
implies that it suffices to prove that $\sigma_{i, \alpha}|_{I_{T_\alpha}} \cong \sigma_{i, \alpha}|_{I_{T_\alpha}}$ if~$\Lambda^+(\sigma_1, \Theta_E) = \Lambda^+(\sigma_2, \Theta_E)$, which is immediate from~(\ref{restrictiontoinertia}).
\end{proof}

By Proposition~\ref{sameinertialparameter}, we have a well-defined injection
\begin{displaymath}
\Lambda^+(-, \Theta_E) : (\text{supercuspidal inertial types of dimension~$n$ over~$R$ containing~$\Theta_F$}) \to \Gamma(\Theta_F) \backslash X_R(\Theta_F)
\end{displaymath}
with image the $\be$-regular orbits. 
The left-hand side consists, of course, of those irreducible representations whose restriction to~$P_F$ corresponds to~$\Theta_F$. 

Now we extend this to certain non-supercuspidal types.
The Langlands parameter of a simple $R[\GL_n(F)]$-representation~$\pi$ has restriction to inertia isomorphic to $\sigma^{\oplus m}$, for some~$m | n$ and some supercuspidal inertial type~$\sigma$ of dimension $n/m$. 
This motivates the following definition.

\begin{defn}
A simple inertial type of endo-class~$\Theta_F$ is a representation of~$R[W_F]$ isomorphic to~$\sigma^{\oplus m}$, for some supercuspidal inertial type~$\sigma$ of endo-class~$\Theta_F$.
\end{defn}

We extend the map~$\Lambda^+$ to simple inertial types of endo-class~$\Theta_F$ and dimension~$n$ by putting
\begin{displaymath}
\Lambda^+(\sigma^{\oplus m}, \Theta_E) = N^* \Lambda^+(\sigma, \Theta_E)
\end{displaymath}
where $N : \be_{n/\delta(\Theta_F)}^\times \to \be_{n/m\delta(\Theta_F)}^\times$ is the norm map.

\begin{rk}
To see that~$n/m\delta(\Theta_F)$ is an integer, one could notice that it equals $\dim \sigmatr$ by the proof of Lemma~\ref{dimensioncomputation}, since $\sigma$ is an $n/m$-dimensional irreducible representation of~$W_F$ of endo-class~$\Theta_F$.
\end{rk}

Finally, because of the statement of~\cite[Types Theorem]{BHeffective}, it will be convenient to twist~$\Lambda^+$ by a certain automorphism of~$\be_{n/\delta(\Theta_F)}^\times$. Let~$p^r$ be the degree of any parameter field~$P$ of~$\Theta_F$ over the maximal tamely ramified extension of~$F$ it contains (this is the degree of the ``wildly ramified part'' of the endo-class~$\Theta_F$). 
\begin{defn} \label{levelzeroLanglandsparameter}
Let~$\tau$ be a simple inertial type of endo-class~$\Theta_F$. Define the level zero part of~$\tau$ by
\begin{displaymath}
\Lambda(\tau, \Theta_E) = \Lambda^+(\tau, \Theta_E)^{p^{-r}}.
\end{displaymath}
\end{defn}

\subsection{Comparison with~$\GL_n(F)$.}\label{comparisonGL_n}
Let~$\kappa$ be a conjugacy class in~$\GL_n(F)$ of $\beta$-extensions of maximal simple characters of endo-class~$\Theta_F$.
Then $\kappa$ together with a lift~$\Theta_E \to \Theta_F$ defines a map~$\Lambda(-, \Theta_E, \kappa)$ on simple inertial classes with endo-class~$\Theta_F$. 
On the other hand, the local Langlands correspondence over~$\bC$ puts the simple $\bC$-inertial classes with endo-class~$\Theta_F$ in bijection with the simple $\bC$-inertial types with endo-class~$\Theta_F$.
We also write~$\rec$ for this bijection.
From now on in this paper we will work with a fixed choice of lift~$\Theta_E \to \Theta_F$, which we will therefore omit from the notation.
The specific choice of lift will not affect the results: what matters here is that the same lifts gets used for~$\GL_n(F)$ and for~$W_F$.

\begin{defn}
Let~$\kappa$ be a maximal $\beta$-extension of endo-class~$\Theta_F$ in~$\GL_n(F)$.
We define a permutation~$\xi(\kappa)$ of the set~$\Gamma(\Theta_F) \backslash X_\bC(\Theta_F)$, depending on~$\kappa$, via
\begin{displaymath}
\xi(\kappa)(\Lambda(\pi, \kappa)) = \Lambda(\rec\,\pi)
\end{displaymath}
for any simple irreducible representation~$\pi$ of~$\GL_n(F)$ with endo-class~$\Theta_F$. 
\end{defn}

The proof of our main result will be based on the following two properties of~$\xi(\kappa)$, one of which involves reduction modulo a prime. 
If~$\ell \ne p$ is a prime number, any isomorphism~$\iota_\ell : \bC \to \cbQ_\ell$ defines a bijection~$\rec_\ell$ analogous to~$\rec$ by means of the resulting identification of simple inertial classes and simple inertial types over~$\bC$ with their analogues over~$\cbQ_\ell$. 
Then the map~$\rec_\ell$ defines a permutation~$\xi_\ell(\kappa)$ of $\Gamma(\Theta_F) \backslash X_{\cbQ_l}(\Theta_F)$ in the same way, and~$\xi_\ell(\kappa)$ is intertwined with~$\xi(\kappa)$ by~$\iota_\ell$.

\begin{lemma}\label{sameparametricdegree}
Define the parametric degree of $[\chi] \in \Gamma(\Theta_F) \backslash X_\bC(\Theta_F)$ as the size of the orbit~$[\chi]$.
Then the map~$\xi(\kappa)$ preserves parametric degrees.
\end{lemma}
\begin{proof}
This is an immediate consequence of the definition of the level zero maps together with the compatibility of~$\rec$ with the Bernstein--Zelevinsky classification.
\end{proof}
 
\begin{thm}\label{samelregularpart}
Let~$\ell \ne p$ be a prime number.
Two elements of~$\Gamma(\Theta_F) \backslash X_\bC(\Theta_F)$ have the same $\ell$-regular part if and only if their images under $\xi(\kappa)$ have the same $\ell$-regular part.
\end{thm}
\begin{proof}
(Compare~\cite[Section~6.2]{Vignerasl1}.)
By the discussion above, it suffices to fix an isomorphism $\iota_\ell: \bC \isom \cbQ_\ell$ and to prove the theorem for~$\xi_{\ell}(\kappa)$ instead of~$\xi(\kappa)$.
Since~$\xi_\ell(\kappa)$ is a bijection, it suffices to prove that it preserves equality of $\ell$-regular parts. 
Consider two simple irreducible integral $\cbQ_\ell[\GL_n(F)]$-representations~$\pi_i$ with endo-class~$\Theta_F$. 
Write $\Lambda(\pi_i, \kappa) = [\psi_i]$, and assume $[\psi_1]^{(\ell)} = [\psi_2]^{(\ell)}$. 

We need to prove that $\xi_\ell(\kappa)[\psi_1]^{(\ell)} = \xi_\ell(\kappa)[\psi_2]^{(\ell)}$, or equivalently that $\Lambda^+(\rec_{\cbQ_\ell}\pi_1)^\pexp{\ell} = \Lambda^+(\rec_{\cbQ_\ell} \pi_2)^{\pexp{\ell}}$.
Assume that~$\psi_i$ is norm-inflated from an $\be$-regular character~$\mu_i$ of~$\be_{n/a_i\delta(\Theta_F)}^\times$. 
By \cite[Lemma~3.20]{InertialJL}, the equality $\Lambda(\br_\ell\pi_i, \br_\ell\kappa) = [\psi_i]^{(\ell)}$ holds. 
(Strictly speaking, we should work with an irreducible factor of~$\br_\ell \pi_i$. However, the irreducible factors of~$\br_\ell \pi_i$ all have the same supercuspidal support, hence there is no ambiguity in writing~$\Lambda(\br_\ell\pi_i, \br_\ell\kappa)$. The same remark applies to the following lemma.)

\begin{lemma}\label{samesupercuspidalsupport}
We can choose the~$\pi_i$ in their inertial class so that the~$\br_\ell(\pi_i)$ have the same supercuspidal support. 
\end{lemma}
\begin{proof}
First choose the~$\pi_i$ so that they have supercuspidal support of the form $ (\pi_i^0)^{\otimes a_i}$ for integral representations~$\pi_i^0$. 
By the classification of cuspidal $\cbF_\ell[\GL_n(F)]$-representations, the supercuspidal support of~$\br_\ell(\pi_i^0)$  has the form $\tau_i \otimes \cdots \otimes \tau_i(m_i-1)$ for some supercuspidal $\tau_i$.
It follows from the fact that $\Lambda(\br_\ell\pi_i, \br_\ell\kappa) = [\psi_i]^{(\ell)}$ together with our assumption on~$[\psi_i]^\pexp{\ell}$ that $\br_\ell(\pi_1)$ and~$\br_\ell(\pi_2)$ are inertially equivalent.
By uniqueness of supercuspidal support, it follows that $a_1m_1 = a_2 m_2$ and~$\tau_1$ is an unramified twist of~$\tau_2$.
So there exist unramified $\cbQ_\ell$-characters~$\chi_i$ such that any two $\cbQ_\ell$-representations~$\pi_i$ with supercuspidal support $\chi_i\pi_i^0 \otimes \chi_i\pi_i^0(m_i) \otimes \cdots \otimes \chi_i\pi_i^0((a_i-1)m_i)$ satisfy the conclusion of the lemma.
(One can take~$\chi_1 = 1$ and then~$\chi_2$ such that $\tau_1 = \br_\ell(\chi_2) \tau_2$.)
\end{proof}

Choose~$\pi_1, \pi_2$ as in Lemma~\ref{samesupercuspidalsupport}.
Write~$\tau_i$ for the semisimple $W_F$-representation underlying~$\rec_{\cbQ_\ell}(\pi_i)$. 
It is a direct sum of~$a_i$ copies of some irreducible representation~$\sigma_i$. 
By \cite[Th\'eor\`eme principal]{Vignerasl2}, our assumption on the~$\pi_i$ implies that $\br_\ell(\tau_1) = \br_\ell(\tau_2)$.

By Theorem~\ref{ramification}, there exists an irreducible representation~$\alpha$ of~$P_F$ that is contained in both~$\sigma_1$ and~$\sigma_2$ and whose $W_E$-orbit corresponds to~$\Theta_E$.
Let~$W_T$ be the stabilizer of~$\alpha$ in~$W_F$ and fix a $\cbQ_\ell$-integral extension~$\rho(\alpha)$ of~$\alpha$ to~$W_T$.
Then~$\sigma_i$ can be written as the induction of its $\alpha$-isotypic component: 
\[
\sigma_i \cong \Ind_T^F (\rho(\alpha) \otimes \sigmatr_i(\alpha)).
\]
There exist integers~$d_i = n/a_i\delta(\Theta_F)$ and tamely ramified characters~$\chi_i = \chi_i(\alpha)$ of~$T_{d_i}^\times$ such that $\sigmatr_i(\alpha) = \Ind_{T_{d_i}}^T \chi_i(\alpha)$ and $\sigma_i = \Ind_{T_{d_i}}^F\rho(\alpha) \otimes \chi_i(\alpha)$. 
By definition, we have $\Lambda^+(\rec_{\cbQ_\ell}\pi_i) = N^*[\chi_i|\mu_{T_{d_i}}]$, where we write~$N: \be_{n/\delta(\Theta_F)}^\times \to \be_{d_i}^\times$ for the norm, and~$\chi_i|\mu_{T_{d_i}}$ is viewed as a character of the residue field of~$T_{d_i}$, which is identified with~$\be_{d_i}$.
So it will suffice to prove that $(\chi_1|{\mu_{T_{d_1}}})^{(\ell)}$ and $(\chi_2|{\mu_{T_{d_2}}})^{(\ell)}$ are both norm-inflated from $\mu_T$-regular characters of the same~$\mu_{T_r}$ for some $r>0$, and that these characters of~$\mu_{T_r}$ are conjugate over~$T$.

Since the wild inertia group~$P_F$ is a pro-$p$ group, we can identify its representations over~$\cbQ_\ell$ and~$\cbF_\ell$, and we will write~$\alpha$ to indicate the representation~$\br_\ell(\alpha)$.
Since~$\rho(\alpha)$ is a $\cbQ_\ell$-integral extension of~$\alpha$ to~$W_T$, its reduction~$\br_\ell(\rho(\alpha))$ is irreducible, and we will also denote it by~$\rho(\alpha)$.
Now we use the fact that $\br_\ell(\sigma_i)$ is the semisimplification of $\Ind_{T_{d_i}}^F (\rho(\alpha) \otimes \br_\ell(\chi_i))$. 
The character~$\xi_i = \br_\ell(\chi_i)$ need not be $\ell$-regular, and it extends to its stabilizer in~$W_T$, which is the Weil group of some intermediate unramified extension~$T_{r_i}$ of~$T$. 
Since~$\rho(\alpha)$ extends to~$W_T$, hence to~$W_{T_{r_i}}$, the induction $\Ind_{T_{d_i}}^{T_{r_i}}(\rho(\alpha) \otimes \xi_i)$ semisimplifies to a direct sum (possibly with multiplicity) of representations of the form $\rho(\alpha) \otimes \widetilde{\xi}_i$, where~$\tld \xi_i$ ranges over extensions of~$\xi_i$ to~$T_{r_i}^\times$. 
All these extensions are unramified twists of each other.

Next we are going to prove that each induced representation $\Ind_{T_{r_i}}^F(\rho(\alpha) \otimes \widetilde{\xi}_i)$ is irreducible.
To do so, observe first that the representation
\[
X_i = \Ind_{T_{r_i}}^T(\rho(\alpha) \otimes \widetilde{\xi}_i) \cong \rho(\alpha) \otimes \Ind_{T_{r_i}}^T(\widetilde{\xi}_i)
\]
is irreducible, since its restriction to~$W_{T_{r_i}}$ is semisimple and has multiplicity one.
Now we have a $P_F$-linear decomposition
\[
\Ind_T^F(X_i) = \bigoplus_{g \in W_F/W_T} gX_i
\]
and since~$W_T$ is the stabilizer of~$\alpha$ in~$W_F$, this is actually the decomposition into isotypic components for~$P_F$.
Since~$gX_i$ is an irreducible representation of~$gW_Tg^{-1}$, it follows that~$\Ind_T^F(X_i)$ is irreducible over~$W_F$.

It follows from the above that $\br_\ell(\sigma_i)$ is a direct sum of unramified twists of a single irreducible representation, which can be taken to be any of the $\Ind_{T_{r_i}}^F(\rho(\alpha) \otimes \widetilde{\xi}_i)$.
Since~$\br_l(\tau_1) = \br_l(\tau_2)$ and $\br_l(\tau_i)$ is a multiple of~$\br_l(\sigma_i)$ in the Grothendieck group, we see that $\Ind_{T_{r_1}}^F (\rho(\alpha) \otimes \widetilde{\xi}_1)$ and $\Ind_{T_{r_2}}^F (\rho(\alpha) \otimes \widetilde{\xi}_2)$ are unramified twists of each other. Comparing dimensions, this implies that~$r_1$ is equal to~$r_2$, and we denote their common value by~$r$. 
Passing to the $\alpha$-isotypic components~$X_i$, we find that the restriction to~$\mu_{T_r}$ of the~$\widetilde{\xi}_i$ are conjugate over~$T$. But since $\xi_i = \br_l(\chi_i)$ this implies that $(\chi_1|{\mu_{T_{d_1}}})^{(l)}$ and $(\chi_2|{\mu_{T_{d_2}}})^{(l)}$ are conjugate over~$T$, after descending to~$\mu_{T_r}$ via the norm.
\end{proof}

\section{Canonical $\beta$-extensions.}
We will work over the complex numbers unless otherwise stated.
Fix an endo-class~$\Theta_F$ of degree dividing~$n$.
We begin by defining the conjugacy class~$\kappa_{n, F}$ of maximal $\beta$-extensions in~$\GL_n(F)$ that appears in the introduction.
To do so, fix a maximal simple character~$\theta$ of endo-class~$\Theta_F$ in~$\GL_n(F)$ and recall that any two $\beta$-extensions of~$\theta$ are twists of one another by a character of~$\be^\times$ inflated through
\[
J_\theta/J^1_\theta \isom \GL_{n/\delta(\Theta_F)}(\be) \xrightarrow{\mathrm{det}} \be^\times.
\]
Since the order of~$\be^\times$ is coprime to~$p$ and the dimension of any $\beta$-extension is a power of~$p$, and~$J^1_\theta$ is a pro-$p$ group, there exists a unique $\beta$-extension of~$\theta$ whose determinant character is $p$-primary (that is to say, has order equal to a power of~$p$).
We will refer to it as the \emph{$p$-primary} $\beta$-extension of~$\Theta_F$ and denote it by~$\kappa_p$.

Now we describe~$\kappa_{n, F}$ as a quadratic twist of~$\kappa_p$.
Let~$\epsilon_\theta^1$ be the symplectic sign character of~$\theta$  (as defined in~\cite[5.4]{BHeffective}). 
Write~$\epsilon_{\Gal}$ for the quadratic character of~$\be^\times$ which is nontrivial if and only if $p \not = 2$ and the ramification degree of a tame parameter field of~$\Theta_F$ over~$F$ is even. 
Then \cite[Types Theorem]{BHeffective} implies that if we set
\[
\kappa_{n, F} = \epsilon_\Gal \epsilon^1_\theta \kappa_p
\]
then
\[
\Lambda(\pi, \kappa_{n, F}) = \Lambda(\rec\, \pi)
\]
for every supercuspidal representation~$\pi$ of~$\GL_n(F)$ with endo-class~$\Theta_F$.
(The reference~\cite{BHeffective} is written in a slightly different language than this paper, and it is a lengthy but routine matter to translate between the two.)
In other words, the permutation~$\kappa_{n, F}$ fixes all the Galois orbits of $\be^\times$-regular characters of~$\be_{n/\delta(\Theta_F)}^\times$.
The formal properties of~$\xi(\kappa_{n, F})$ given in Section~\ref{comparisonGL_n} now allow us to deduce that~$\xi(\kappa_{n, F})$ is the identity.
Namely, we have the following theorem.

\begin{thm}\label{canonicalbetaextensions}
Assume that for every~$t>0$ there is a maximal $\beta$-extension~$\kappa_{tn, F}$ of endo-class~$\Theta_F$ in~$\GL_{tn}(F)$ such that~$\xi(\kappa_{tn, F})[\xi] = [\xi]$ whenever~$\xi$ is a $\be$-regular character of~$\be_{tn/\delta(\Theta_F)}^\times$.
Then~$\xi(\kappa_{tn, F})$ is the identity for all~$t$.
\end{thm}
\begin{proof}
This is proved using a technique introduced in \cite[Lemma~9.11]{SecherreStevensJL}. 
Replacing~$n$ by~$tn$, it suffices to prove the theorem when~$t =1$.
We write~$\kappa$ for~$\kappa_{n, F}$.
Assume that~$\alpha$ is a character of~$\be_{n/\delta(\Theta_F)}^\times$ which is not $\be$-regular. 
We need to prove that $\xi(\kappa)[\alpha] = [\alpha]$. 
Consider a simple representation~$\pi$ of~$\GL_n(F)$ with supercuspidal support~$\pi_0^{\otimes r}$ such that $\Lambda(\pi, \kappa) = [\alpha]$.

Let~$a \geq 1$ be some large integer ($a \geq 7$ will suffice) and write~$\kappa^*$ for the maximal $\beta$-extension in~$\GL_{an}(F)$ compatible with~$\kappa$. 
Let~$\pi_a$ be a representation of~$\GL_{an}(F)$ with supercuspidal support~$\pi_0^{\otimes ar}$. 
Since compatible $\beta$-extensions satisfy a transitivity property (see~\cite[Proposition~3.16]{InertialJL}) we know that $\Lambda(\pi_a, \kappa^*)$ is the inflation~$[\alpha^*]$ of~$\alpha$ to~$\be_{an/\delta(\Theta_F)}^\times$. 

By~(\ref{Langlandssupport}) we have $\rec(\pi_a)|_{I_F} = \rec(\pi_0)|_{I_F}^{\oplus ar}$, so that if $\Lambda(\rec \, \pi) = [\mu]$ then $\Lambda(\rec \, \pi_a) = [\mu^*]$. 
Hence by definition we have $[\mu] = \xi(\kappa)[\alpha]$ and $[\mu^*] = \xi(\kappa^*)[\alpha^*]$ (although at this stage we do not know whether $[\alpha] = [\mu]$). 
So it suffices to prove that $\xi(\kappa^*)[\alpha^*] = [\alpha^*]$: since the norm is surjective in finite extensions of finite fields, this will imply that~$[\alpha] = [\mu]$.
Notice that it also follows that $\xi(\kappa^*)[\alpha^*] = (\xi(\kappa)[\alpha])^*$, which will be useful later in the proof. 

Write $\be[\alpha]^\times$ for the fixed field of the stabilizer of~$\alpha$ in~$\Gal(\be_{n/\delta(\Theta_F)} / \be)$. 
By~\cite[Lemma~8.5, Remark~8.7]{SecherreStevensJL}, there exist an $\be$-regular character~$\beta$ of~$\be_{an/\delta(\Theta_F)}^\times$ and a prime number~$\ell \not = p$ not dividing the order of~$\be[\alpha]^\times$ such that $\alpha^*$ is the $\ell$-regular part of~$\beta$. 
By Proposition~\ref{samelregularpart} we have that $(\xi(\kappa^*)[\alpha^*])^{(\ell)} = (\xi(\kappa^*)[\beta])^{(\ell)}$, and it suffices now to prove that $\xi(\kappa^*)[\beta] = [\beta]$ and that $\xi(\kappa^*)[\alpha^*]$ is $\ell$-regular. 
That $\xi(\kappa^*)[\alpha^*]$ is $\ell$-regular follows by Proposition~\ref{sameparametricdegree}, because it has the same parametric degree as~$[\alpha^*]$ and~$\ell$ does not divide the order of~$\be[\alpha]^\times$.

By assumption, there exists some $\beta$-extension~$\kappa_{an, F}$ in~$\GL_{an}(F)$ such that $\xi(\kappa_{an, F})[\beta] = [\beta]$. 
So there exists some character~$\delta$ of~$\be^\times$ such that $\xi(\kappa^*)[\beta] = [\delta\beta]$ for \emph{every} $\be$-regular character~$\beta$ of~$\be_{an/\delta(\Theta_F)}^\times$, because~$\kappa_{an, F}$ and~$\kappa^*$ are $\be^\times$-twists of each other. 
We will prove that~$\delta$ is trivial: this implies the theorem.

Fix an $\be$-regular character~$\alpha_+$ of~$\be_{n/\delta(\Theta_F)}^\times$. 
Because~$a$ is large enough, again by~\cite[Lemma, 8.5, Remark~8.7]{SecherreStevensJL} there exist some prime number $q \not = p$ not dividing the order of~$\be_{n/\delta(\Theta_F)}^\times = \be[\alpha_+]^\times$ and some $\be$-regular character $\beta_+$ of $\be_{an/\delta(\Theta_F)}^\times$ such that~$\alpha^*_+$ is the $q$-regular part of~$\beta_+$.

We know that $\xi(\kappa)[\alpha_+] = [\alpha_+]$ by regularity of~$\alpha_+$. 
On the other hand, $\xi(\kappa^*)[\alpha^*_+] = [\delta \beta_+]^{(q)} = [\delta^{(q)}\alpha^*_+]$, and since $\xi(\kappa^*)[\alpha^*_+] = (\xi(\kappa)[\alpha_+])^*$ we find that $[\alpha^*_+] = [\delta^{(q)} \alpha^*_+]$.
Since~$\delta$ is a character of~$\be^\times$, and~$q$ does not divide the order of~$\be_{n/\delta(\Theta_F)}^\times$, we also know that~$\delta^\pexp{q} = \delta$. 
So we can write $\delta =(\alpha^*_+)^{|\be|^{i}-1}$ for some integer $i \in \{0, \ldots, \frac{n}{\delta(\Theta_F)} - 1 \}$. 

Now we can take~$\alpha_+$ to be a generator of the character group of $\be_{n/\delta(\Theta_F)}^\times$, hence we can assume that~$\alpha_+$ has order $|\be|^{n/\delta(\Theta_F)} - 1$. 
However, the equality $\delta =(\alpha^*_+)^{|\be|^{i}-1}$ implies that the order of~$\alpha^*_+$ divides $(|\be|^i - 1)(|\be| - 1)$.
Since $|\be| \geq 2$ we have $|\be|^{n/\delta(\Theta_F)} - 1 > (|\be|^i - 1)(|\be| - 1)$, hence $i = 0$ and~$\delta$ is trivial.
\end{proof}

\begin{rk}\label{uniquenessbeta}
Since any two $\beta$-extensions of the same maximal simple character are twists of each other by a character of~$\be^\times$, the property that~$\xi(\kappa_{n, F})$ is the identity determines~$\kappa_{n, F}$ uniquely.
Actually, by~\cite[Proposition~3.6]{InertialJL}, the fact that~$\xi(\kappa_{n, F})$ fixes the $\be$-regular characters already suffices to determine~$\kappa_{n, F}$ uniquely.
\end{rk}

It is now a simple matter to prove our main result, namely that the canonical $\beta$-extensions~$\kappa_{n, F}$ form a compatible family as~$n$ varies amongst multiples of~$\delta(\Theta_F)$.

\begin{thm}\label{compatiblecanonical}
Let~$\Theta_F$ be an endo-class over~$F$, and fix a positive integer~$n$.
Then~$\kappa_{n, F}$ is compatible with~$\kappa_{tn, F}$ for each positive integer~$t$.
\end{thm}
\begin{proof}
By~\cite[Proposition~3.16]{InertialJL} it suffices to prove that~$\kappa_{n, F}$ is compatible with~$\kappa_{\delta(\Theta_F), F}$, which is a $\beta$-extension in~$\GL_{\delta(\Theta_F)}(F)$. 
Write~$\kappa_{\delta(\Theta_F), F}^+$ for the $\beta$-extension in~$\GL_n(F)$ compatible with~$\kappa_{\delta(\Theta_F), F}$. 
Let~$\pi_{\delta(\Theta_F)}$ be a supercuspidal representation of~$\GL_{\delta(\Theta_F)}(F)$ with endo-class~$\Theta_F$ such that $\Lambda(\pi_{\delta(\Theta_F)}, \kappa_{\delta(\Theta_F), F}) = [1]$. 
There exists a character~$\chi$ of~$\be^\times$ such that~$\chi\kappa_{n, F} \cong \kappa_{\delta(\Theta_F), F}^+$, and then $\Lambda(\pi, \kappa_{n, F}) = \chi\Lambda(\pi, \kappa_{\delta(\Theta_F), F}^+)$ for all simple representations~$\pi$ of endo-class~$\Theta_F$.

Let~$\pi$ be a simple representation of~$\GL_{n}(F)$ with supercuspidal support inertially equivalent to $\pi_{\delta(\Theta_F)}^{\otimes n/ \delta(\Theta_F)}$. 
Then $\Lambda(\pi, \kappa_{\delta(\Theta_F), F}^+) = [1]$, by compatibility of~$\kappa_{\delta(\Theta_F), F}$ and~$\kappa_{\delta(\Theta_F), F}^+$. 
Since~$\xi(\kappa_{\delta(\Theta_F), F})$ is the identity, we know that $\Lambda(\rec \, \pi_{\delta(\Theta_F)}) = \Lambda(\pi_{\delta(\Theta_F)}, \kappa_{\delta(\Theta_F), F}) = [1]$.
But we also know that $\Lambda(\pi, \kappa_{n, F}) = \Lambda(\rec \, \pi)$, since~$\xi(\kappa_{n, F})$ is the identity.
By construction of the level zero map for Langlands parameters, we have that $\Lambda(\rec \, \pi)$ is inflated from $\Lambda(\rec \, \pi_{\delta(\Theta_F)})$, hence $\Lambda(\pi, \kappa_{n, F}) = [1]$. 
It follows that~$\chi = 1$, hence~$\kappa_{n, F}$ is compatible with~$\kappa_{\delta(\Theta_F), F}$.
\end{proof}

\subsection{The case of~$\GL_m(D)$.}\label{innerforms}
Using the Jacquet--Langlands correspondence we can now construct canonical maximal $\beta$-extensions in all inner forms~$\GL_m(D)$ of~$\GL_n(F)$.
We refer to~\cite{InertialJL} for analogues of the constructions in Section~\ref{invariantsG} for the group~$\GL_m(D)$.

\begin{thm}\label{betainnerforms}
Let~$D$ be a central division algebra over~$F$ of dimension~$d^2$ and let~$m$ be a positive integer such that~$md = n$. 
Let~$\Theta_F$ be an endo-class over~$F$ of degree dividing~$n$.
Then there exists a unique conjugacy class~$\kappa_{m, D}$ of maximal $\beta$-extensions in~$\GL_m(D)$ of endo-class~$\Theta_F$ such that
\[
\Lambda(\pi, \kappa_{m, D}) = \Lambda(\JL(\pi), \kappa_{n ,F})
\]
for all essentially square-integrable representations~$\pi$ of endo-class~$\Theta_F$, where~$\JL(\pi)$ denotes the Jacquet--Langlands transfer of~$\pi$ to~$\GL_n(F)$.
\end{thm}
\begin{proof}
The uniqueness part follows as in the case of~$\GL_n(F)$, see Remark~\ref{uniquenessbeta}.
Let~$\theta$ be a maximal simple character in~$\GL_m(D)$ of endo-class~$\Theta_F$.
Define a $\beta$-extension of~$\theta$ by setting
\[
\kappa_{m, D} = \epsilon_\Gal\epsilon_\theta^1\kappa_p
\]
where $\kappa_p$ is the $p$-primary $\beta$-extension and $\epsilon_{\Gal}$ is a quadratic character that is nontrivial if and only if $p \not = 2$ and the ramification degree of a tame parameter field of~$\Theta_F$ over~$F$ is even. 
Then the theorem is an immediate corollary of the main results of~\cite{InertialJL}, which imply that $\cl(\JL(\pi)) = \Theta_F$ and
\[
\Lambda(\pi, \epsilon_\Gal\kappa_{m, D}) = \Lambda(\JL(\pi), \epsilon_\Gal \kappa_{n, F}).
\]
\end{proof}

\bibliographystyle{amsalpha}
\bibliography{refpapers}

\end{document}